\documentclass[a4paper,leqno]{mathscan}

\usepackage{color}

\newtheorem{theorem}{Theorem}[section]
\newtheorem{proposition}[theorem]{Proposition}
\newtheorem{definition}[theorem]{Definition}
\newtheorem{remark}[theorem]{Remark}
\newtheorem{lemma}{Lemma}
\newtheorem{corollary}{Corollary}     

\theoremstyle{definition}           

\newcommand{\R}{\mathbb{R}}
\newcommand{\N}{\mathbb{N}}

\newcommand{\C}{\mathbb{C}}
\newcommand{\Nab}[3]{{\overset{{#2}}{\underset{{#1}}{\displaystyle{\nabla}}}}{\hspace*{-.13cm}\phantom{i}^{{#3}}}}
\newcommand{\Ss}[3]{{\overset{{#2}}{\underset{{#1}}{\displaystyle{S}}}}{\hspace*{-.13cm}\phantom{i}^{{#3}}}}

\newcommand{\hypergeom}[5]{\mbox{$
_#1 F_#2\left. \!\! \left( \!\!\!\!
\begin{array}{c}
\multicolumn{1}{c}{\begin{array}{c} #3
\end{array}}\\[1mm]
\multicolumn{1}{c}{\begin{array}{c} #4
            \end{array}}\end{array}
\!\!\!\! \right| \displaystyle{#5}\right) $} }

\begin{document}

\title[]{Continuous and discrete fractional operators  and some fractional functions  }            

\author{P. Njionou Sadjang$^1$, S. Mboutngam$^2$}
\thanks{}

\address{$^1$Faculty of Industrial Engineering, University of Douala, Douala, Cameroon \newline
$^2$Higher Teachers' Training College,
  University of Maroua, Maroua, Cameroon}

\email{$^1$ pnjionou@yahoo.fr, $^2$ mbsalif@gmail.com}



\thanks{}

\begin{abstract}
The classical orthogonal polynomials are usually defined by the
Rodrigues' formula. This paper refers to a fractional
extension of the classical Hermite, Laguerre, Jacobi, Charlier, Meixner, Krawtchouk and Hahn polynomials.  By means of the
Caputo operator of fractional calculus, C-Hermite, C-Laguerre, C-Legndre and the C-Jacobi
functions are defined and their representation in terms of the
hypergeometric functions are provided. Also, by means of the
Gray and Zhang fractional difference oparator, fractional Charlier, Meixner, Krawtchouk and Hahn functions are defined and their representation in terms of the
hypergeometric functions are provided. Some other properties of the new defined functions are given.
\end{abstract}

\maketitle


\section{Introduction}

\noindent Fractional calculus is the field of mathematical analysis which deals with the investigation and applications of derivatives and integrals of arbitrary (real or complex) order. It is a complex and interesting topic having interconnections with various problems of function theory, integral and differential equations, and other branches of analysis. It has been continually developped, sitmulated by ideas and results in various fields of mathematical analysis. This is demonstrated by the many pubilcations--hundreds of papers in the past years--and by the many conferences devoted to the problems of fractional calculus. 

A sequence of polynomials $\{p_n(x)\}$, where $p_n(x)$ is of exact degree $n$ in $x$, is said to be orthogonal with respect to a Lebesgue-Stieltjes measure $d\alpha(x)$ if
\begin{equation}\label{ortho1}
\int_{-\infty}^{\infty}p_m(x)p_n(x)d\alpha(x)=0,\quad m\neq n.
\end{equation}
Implicit in this definition is the assumption that the moments
\begin{equation}\label{moment}
\mu_n=\int_{-\infty}^{\infty}x^nd\alpha(x),\quad n=0,1,2,\ldots,
\end{equation}
are finite. If the nondecreasing, real-valued, bounded function $\alpha(x)$  is a step-function with jumps $\rho_j$ at $x=x_j$, $j=0,1,2,\ldots$, then (\ref{ortho1}) and (\ref{moment}) take the form of a sum:
\begin{equation}
\sum_{j=0}^{\infty}p_m(x_j)p_n(x_j)\rho_j=0,\quad m\neq n
\end{equation}
and
\begin{equation}
\mu_n=\sum_{j=0}^{\infty}x_j^n\rho_j,\quad n=0,1,2,\ldots.
\end{equation}

A polynomial set

\begin{equation}\label{cop1}
y(x)=p_n(x)=k_nx^n+\dots \quad (n\in\N_0=\{0,1,2,\dots\},\;\;k_n\neq 0)
\end{equation}
is a family of {\sl classical continuous orthogonal polynomials}  if it is the solution of a differential equation of the type

\begin{equation}
\sigma (x) y''(x)+\tau(x)y'(x)+\lambda_n y(x)=0,
\end{equation}
where $\sigma(x)=ax^2+bx+c$ is a polynomial of at most second order and $\tau(x)=dx+e$ is a polynomial of first order. Here, the distribution $d\alpha(x)$ takes the form
\[d\alpha(x)=\rho(x)dx,\]where $\rho$ is the non negative solution on $(a,b)$ of the Pearson equation
\[\dfrac{d}{dx}(\sigma(x)\rho(x))=\tau(x)\rho(x).\]  Up to a linear change of variable,
these polynomials can be classified as (the hypergeometric representation $_pF_q$ is defined in Section \ref{sec:2}):
\begin{description}
\item[(a)] The Jacobi polynomials {\cite[P.\ 216]{KLS2010}}
\[P_n^{(\alpha,\beta)}(x)=\hypergeom{2}{1}{-n,n+\alpha+\beta+1}{\alpha+1}{\dfrac{1-x}{2}}.\]
\item[(b)] The Laguerre polynomials \cite[P.\ 241]{KLS2010}
\[L_n^{(\alpha)}(x)=\frac{(\alpha+1)_n}{n!} \hypergeom{1}{1}{-n}{\alpha+1}{x}.\]
\item[(c)] The Hermite polynomials \cite[P.\ 250]{KLS2010}
\[H_n(x)=(2x)^n \hypergeom{2}{0}{-\frac{n}{2},-\frac{n-1}{2}}{-}{-\frac{1}{x^2}}.\]
\end{description}

\noindent Some special cases of the Jacobi polynomials are:
\begin{itemize}
 \item[{\bf (a-1)}] The Gegenbauer / Ultraspherical polynomials \cite[P.\ 222]{KLS2010}\\
 They are Jacobi polynomials for $\alpha=\beta=\lambda-\frac{1}{2}$.
 \begin{eqnarray*}
 C_n^{(\lambda)}&=&\frac{(2\lambda)_n}{\left(\lambda+\frac{1}{2}\right)_n}P_n^{\left(\lambda-\frac{1}{2},\lambda-\frac{1}{2}\right)}(x)\\
 &=& \frac{(2\lambda)_n}{n!}\hypergeom{2}{1}{-n,n+2\lambda}{\lambda+\frac{1}{2}}{\frac{1-x}{2}},\quad \lambda\neq 0.
 \end{eqnarray*}
 \item[{\bf (a-2)}] The Chebyshev polynomials \cite[P.\ 225]{KLS2010}\\
 The Chebyshev polynomials of the first kind can be obtained from the Jacobi polynomials by taking $\alpha=\beta=-\frac{1}{2}$:
 \[T_n(x)=\frac{P_n^{\left(-\frac{1}{2},-\frac{1}{2}\right)}(x)}{P_n^{\left(-\frac{1}{2},-\frac{1}{2}\right)}(1)}=\hypergeom{2}{1}{-n,n}{\frac{1}{2}}{\frac{1-x}{2}},\]
 and the Chebyshev polynomials of the second kind can be obtained from the Jacobi polynomials by taking $\alpha=\beta=\frac{1}{2}$:
 \[U_n(x)=(n+1)\frac{P_n^{\left(\frac{1}{2},\frac{1}{2}\right)}(x)}{P_n^{\left(\frac{1}{2},\frac{1}{2}\right)}(1)}=(n+1)\hypergeom{2}{1}{-n,n+2}{\frac{3}{2}}{\frac{1-x}{2}}.\]
 \item[{\bf (a-3)}] The Legendre polynomials\\
 They are Jacobi polynomials with $\alpha=\beta=0$:
 \[P_n(x)=P_n^{(0,0)}(x)=\hypergeom{2}{1}{-n,n+1}{1}{\frac{1-x}{2}}.\]
\end{itemize}
Also, these polynomials can  be represented by a Rodrigues type formula (see \cite[page 64]{KLS2010})
\begin{equation}
p_n(x)=\dfrac{K_n}{\rho(x)}\dfrac{d^n}{dx^n}\left(\rho(x) \sigma^n(x)\right).
\end{equation}
It should be noted that this representation caracterizes fully the family $p_n$ and then is sometimes used as its definition.

Also, a polynomial set \eqref{cop1}  is a family of discrete classical orthogonal polynomials (also known as the Hahn class) if it is the solution of a difference equation of the type

\begin{equation}
\sigma(x)\Delta\nabla y(x)+\tau(x)\Delta y(x)+\lambda_n y(x)=0,
\end{equation} 
Here the polynomials $\sigma(x)$ and $\tau(x)$ are known to satisfy a Pearson type equation 
\[\Delta(\sigma(x)\rho(x))=\tau(x)\rho(x),\]
where the function $\rho(x)$ is the discrete weight function associated to the family.  
These polynomials can be classified as (see e.g \cite{Koepf-Schmersau2}):
\begin{description}
\item[(d)] The Hahn polynomials \cite[page 204]{KLS2010}
\[Q_n(x;\alpha,\beta,N)= \hypergeom{3}{2}{-n,n+\alpha+\beta+1,-x}{\alpha+1,-N}{1}.\]

\item[(e)] The Krawtchouk polynomials \cite[page 237]{KLS2010}
\[K_n(x;p,N)= \hypergeom{2}{1}{-n,-x}{-N}{\frac{1}{p}} .\]
\item[(f)] The Meixner polynomials \cite[page 234]{KLS2010}
\[M_n(x;\beta,c)= \hypergeom{2}{1}{-n,-x}{\beta}{1-\frac{1}{c}}.\]
\item[(g)] The Charlier polynomials \cite[page 247]{KLS2010}
\[C_n(x;\alpha)=\hypergeom{2}{0}{-n,-x}{-}{-\frac{1}{\alpha}} .\]
\end{description}

\noindent Also, these polynomials can  be represented by a Rodrigues type formula (see \cite[page 71]{KLS2010})
\begin{equation}
p_n(x)=\dfrac{K_n}{\rho(x)}\Delta^n\left(\rho(x-n)\prod_{k=1}^{n} \sigma(x-k-1)\right),
\end{equation}
where $K_n$ is given by
\[K_n=\dfrac{1}{\prod_{k=1}^n(e(2n-k-1)+d)}.\]
It should be noted that this representation caracterizes fully the family $p_n$ and then is sometimes used as its definition.

In \cite{ishteva}, the authors defined the $C$-Laguerre functions from the Rodrigues representation of the Laguerre polynomials by replacing the ordinary derivative by a fractional type derivative, then they gave several properties of the new defined functions. Following their idea, we do the same for all the classical continuous and classical discrete orthogonal polynomials listed above.  It should be noted that the Caputo fractinal derivative applies only when the Rodrigues representation of the polynomial set uses de classical derivative as done for the classical continuous polynomials. In the case of the discrete families, we cannot apply this fractional derivative. In \cite{diaz}, the authors defined the fractional difference by the rather naturel approach of allowing the index of differencing, in the standard expression for the $n$th difference, to be any real or complexe number, that is
\begin{equation}
\Delta^{\alpha}f(x)=\sum_{k=0}^{\infty}(-1)^k\binom{\alpha}{k}f(x+\alpha-k),
\end{equation}
where $\alpha$ is any real or complex number. In \cite{granger}, the author employed the following definition 
\begin{eqnarray}\label{grangerdef}
\nabla^\alpha f(x)= \sum_{k=0}^{\infty}(-1)^k\binom{\alpha}{k}f(x-k),
\end{eqnarray}
and showed that it can be used to study long-memory time series. In \cite{gray}, Gray and Zhang gave a new definition of the fractional difference which also includes the notion of the fractional sum over specific index set. One of the more important features of this definition is that the sum corresponding to the one in \eqref{grangerdef} is finite. In this paper, we make use of the definition of the fractional difference given in \cite{gray}.

The paper is organised as follows:
\begin{enumerate}
    \item In Section 2, we present the preliminary results and definitions that are useful for a better reading of this manuscript.
    \item In Section 3, we introduce the fractional Hermite, Laguerre, Jacobi, Charlier, Meixner, Krawtchouk and  Hahn functions and provide several properties of these functions such as their hypergeometric representation, the differential equations they satisfy$\cdots$
\end{enumerate}

\section{Preliminary definitions and results}\label{sec:2}
{ \fontfamily{times}\selectfont
 \noindent

\subsection{The hypergeometric series and the Gamma function}
{ \fontfamily{times}\selectfont
 \noindent \noindent In what follows, the symbol $(a)_n$ denotes the so-called Pochhammer symbol and is defined by

\[(a)_m=\left\{%
\begin{array}{ll}
    1 \quad \textrm{if }\;\; m=0 \\
    a(a+1)\dots(a+m-1)\quad\textrm{if }\;\; m=1,2,\dots \\
\end{array}%
\right.    \]
 and  the  hypergeometric series is defined as
\begin{equation*}
\hypergeom{p}{q}{a_1,\cdots, a_p}{b_1,\cdots,b_q}{x}=\sum_{n=0}^{\infty}\frac{(a_1)_n\cdots (a_p)_n}{(b_1)_n\cdots (b_q)_n}\frac{x^n}{n!}.
\end{equation*}

\color{black}{\noindent Note that the Pochhammer fulfils the following
Legendre duplication formula \cite[page 22]{rainville}
\begin{equation}
(a)_{2n}=2^{2n}\left(\frac{a}{2}\right)_{n}\left(\frac{1+a}{2}\right)_n.
\end{equation}}\color{black}

\begin{proposition}\cite[page 10]{KLS2010}
The following equation applies:
\begin{eqnarray}
\hypergeom{2}{1}{a,b}{c}{z}&=&(1-z)^{-a}\hypergeom{2}{1}{a,c-b}{c}{\frac{z}{z-1}}.\label{e3}
\end{eqnarray}
\end{proposition}


\begin{proposition}\cite[page 42]{wilhelm}
The Gauss hypergeometric function  $\hypergeom{2}{1}{a,b}{c}{x}$ satifies the differential equation
\begin{equation}\label{equadiff}
x(1-x)\dfrac{d^2}{dx^2}y(x)+[c-(a+b+1)x]\dfrac{d}{dx}y(x)+ab y(x)=0.
\end{equation}
\end{proposition}

\begin{definition}  The Gamma function  is defined by
\[
\Gamma(z)=\int_0^{+\infty}t^{z-1}e^{-t}dt,\;\;\;\forall z\in \R.
\]
\end{definition}
\noindent The Gamma function  satisfies the following
fund\color{black}{a}\color{black}mental properties:
\begin{eqnarray}
\Gamma(z+1)=z\Gamma(z)\label{e1a}\\
\frac{\Gamma(z+k)}{
\Gamma(z)}=(z)_{\color{black}{k}\color{black}}\label{e1b}.
\end{eqnarray}
By \eqref{e1b} the binomial coefficients can be written in terms of the $\Gamma$-function as
\begin{equation}
\binom{z}{k}=\dfrac{z(z-1)\cdots (z-k+1}{k}=\dfrac{\Gamma(z+1)}{k!\Gamma(z-k+1)},
\end{equation}
for arbitrary $z\in \C$, $z+1\neq 0,\ -1,\ldots, $ and $z-k+1\neq 0,\ -1,\ldots,$.

Note, further, the following relation between the Pochhammer symbol and the binomial coefficients,
\begin{equation}
\binom{z}{k}=(-1)^k\binom{k-z-1}{k}=\dfrac{(-1)^k}{k!}(-z)_k.
\end{equation}

\noindent Another function which is closely connected to the Gamma function is the Beta function. The later function is defined by
\begin{equation}
B(p,q)=\int_{0}^{1}t^{p-1}(1-t)^{q-1}\color{black}{dt}\color{black},\quad
p,q\in\C,\quad \Re(p)>0,\; \Re(q)>0.
\end{equation}
The Beta function is connected to the Gamma function by the formula
\begin{equation}
B(p,q)=\dfrac{\Gamma(p)\Gamma(q)}{\Gamma(p+q)},\quad \Re(p)>0,\; \Re(q)>0.
\end{equation}

\subsection{Riemann-Liouville and Caputo differential operators of fractional calculus }
{ \fontfamily{times}\selectfont
 \noindent \begin{definition}
The left-sided Riemann-Liouville fractional integral of order $\nu$ of the fonction $f(t)$ is defined as
\begin{equation}
_a J^{\nu}f(x)=\dfrac{1}{\Gamma(\alpha)}\int_a^{x}(x-\tau)^{\nu-1}f(\tau)d\tau,\quad x>a.
\end{equation}
\end{definition}

\begin{definition}(See \cite[page 68]{podlubny})\label{def1}
Let $\nu>0$ and $x>a$, $\nu,\; a,\; x\in \R$. The Riemann-Liouville differential operator
of fractional calculus of order $\nu$ is defined by
 \begin{equation}\label{e1}
_aD^{\nu}f(x):=\left\{\begin{array}{ll}
                      \dfrac{1 }{\Gamma(n-\nu)}\dfrac{d^n}{ dx^n}\left(\displaystyle{\int_a^x} {f(\tau)\over (x-\tau)^{\nu+1-n}}d\tau\right), & n-1<\nu<n\in \N \\
                    \dfrac {d^n}{ dx^n}f(x), & \nu=n\in\N.
                    \end{array}\right.
 \end{equation}
\end{definition}

\begin{remark}\label{rq1}
If we set $D=\dfrac{d}{dx}$, then it is easy to see that $_aD^{\nu}=D^n\ _aJ^{n-\nu}$, where $n-1\leq \nu \leq n$.
\end{remark}

 \begin{definition}(See \cite[page 79]{podlubny})\label{def2}
Let $\nu>0$ and $x>a$, $\nu,\; a,\; x\in \R$. The Caputo fractional derivative or Caputo differential operator
of fractional calculus of order $\nu$ is defined by \footnote{See the misprint: in this book, it is writen $\dfrac{1}{\Gamma(\nu-n)}$ instead of $\dfrac{1}{\Gamma(n-\nu)}$.} 
 \begin{equation}\label{e2}
_aD^{\nu}_*f(x):=\left\{\begin{array}{ll}
                      \dfrac{1}{\Gamma(n-\nu)}\displaystyle{\int_a^x} \dfrac{ f^{(n)}(\tau)}{ (x-\tau)^{\nu+1-n}}d\tau, & \;\; n-1<\nu<n\in \N \\
                    \dfrac {d^n}{ dx^n}f(x), & \nu=n\in\N.
                    \end{array}\right.
 \end{equation}
\end{definition}
This opertor is introduced by the Italian mathematician M. Caputo in 1967. For the seek of simplicity, we will denote $_0D^{\nu}_*f(x)$ simply by $D^{\nu}_*f(x)$.

We have the following results.

\begin{lemma}\label{lem2}\cite{ishteva, podlubny}
Let $n-1<\nu<n\in \N,\;\; \nu,\; a,\;x\in\R, \;x>a.$ The following relation between the Riemann-Liouville
(\ref{e1}) and the Caputo (\ref{e2}) differential operators holds:

\begin{equation}\label{e6}
_aD^{\nu}_*f(x)=\ _aD^{\nu} f(x)-\sum_{k=0}^{n-1}\frac{f^{(k)}(a) }{\Gamma(k+1-\nu)}(x-a)^{k-\nu}.
\end{equation}
\end{lemma}

\begin{lemma}\cite[page 96]{podlubny}\label{lem20}
If $f(\tau)$ and $g(\tau)$ and all its derivatives are
 continuous in $[a,x]$ then the following relation holds
 \begin{equation}\label{eq001}
 _a D^{\nu}(f(x)g(x))=\sum_{k=0}^{\infty}\binom{\nu}{k}\left( _aD^{\nu-k}f(x)\right)g^{(k)}(x).
 \end{equation}
\end{lemma}

\noindent Puting together Lemma \ref{lem2} and \ref{lem20}, the following proposition arises.

\begin{proposition}\label{lem3}\cite{ishteva, podlubny}
Let $n-1<\nu<n\in \N,\;\; \nu,\;x\in\R, \;x>0.$ If $f(\tau)$ and $g(\tau)$ and all its derivatives are
 continuous in $[a,x]$ then the following relation holds
\begin{eqnarray}\label{e7}
_aD^{\nu}_*(f(x)g(x))&=&\sum_{k=0}^\infty \binom{\nu}{k}\left( _aD^{\nu-k}f(x)\right)g^{(k)}(x)\nonumber \\
&& \hfill-\sum_{k=0}^{n-1}{(x-a)^{k-\nu}\over \Gamma(k+1-\nu)}\left((f(x)g(x))^{(k)}(a)\right).
\end{eqnarray}
\end{proposition}

\noindent The following result which can be deduced from Proposition \ref{lem3}, appears in  \cite[page 31]{Kai}.
\begin{corollary}\cite{Kai}
Let $g$ be analytic in $(a-h,a+h)$ for some $h>0$ and $\nu>0$, $\nu\notin\N$, then
\begin{equation}\label{eq004}
_a D^\nu g(x)=\sum_{k=0}^{\infty}\dfrac{(x-a)^{k-\nu}}{\Gamma(k+1-\nu)}D^k g(a), \quad a<x<a+h.
\end{equation}
\end{corollary}


\noindent The following proposition appears as an example in \cite[page 20]{Kai}.
\begin{proposition} \label{pro25}For $\beta>-1$ and $x>a$ the following relation is valid.
\begin{equation}\label{exemple01}
_a J^{\nu} (x-a)^{\beta}=\dfrac{\Gamma(\beta+1)}{\Gamma(\nu+\beta+1)}(x-a)^{\nu+\beta}.
\end{equation}
\end{proposition}

\begin{proof} From the definition of $_aJ^{\nu}$, and the change of the variable $t=a+s(x-a)$, we have
\begin{eqnarray*}
_aJ^{\nu}(x-a)^{\beta}&=&\int_a^x(t-a)^\beta(x-t)^{\nu-1}dt\\
&=& \dfrac{(x-a)^{\nu+\beta}}{\Gamma(\nu)}\int_0^1 s^{\beta}(1-s)^{\nu-1}ds\\
&=&\dfrac{\Gamma(\beta+1)}{\Gamma(\nu+\beta+1)}(x-a)^{\nu+\beta}.
\end{eqnarray*}
\end{proof}


\noindent The use of Remark \ref{rq1} and equation \eqref{exemple01} leads to the following proposition.
\begin{proposition} \cite{podlubny}\label{pro26}
Let $\beta>\nu-1$ and $t>a$. The following result holds.
\begin{equation}
_aD^\nu [(x-a)^{\beta}]=\dfrac{\Gamma(1+\beta)}{\Gamma(1+\beta-\nu)}(x-a)^{\beta-\nu},
\end{equation}
\begin{equation}
_aD^\nu_* [(x-a)^{\beta}]=\dfrac{\Gamma(1+\beta)}{\Gamma(1+\beta-\nu)}(x-a)^{\beta-\nu}.
\end{equation}
\end{proposition}

\begin{proof}
The proof follows from the definition of $_aD^\nu$ and Proposition \ref{pro25}.
\end{proof}

%

\begin{corollary}
The following result is valid.
\begin{equation}
_0D^\nu_* [x^{2n}]=\dfrac{\Gamma(2n+1)}{\Gamma(2n-\nu+1)}x^{2n-\nu}.
\end{equation}
\end{corollary}

\begin{proof}
Take $a=0$ and $\beta=2n$ in Proposition \ref{pro26}.
\end{proof}

\begin{theorem}\label{theorem00}
For $x>0$, we have:
\begin{equation}
_0D^\nu_* e^{-x^2}=\dfrac{x^{-\nu}}{\Gamma(1-\nu)}\hypergeom{2}{2}{\frac{1}{2},1}{\frac{1-\nu}{2},\frac{2-\nu}{2}}{-x^2}.
\end{equation}
\end{theorem}

\begin{proof}  From the definition of $_0D^\nu$, it follows that:
\begin{eqnarray*}
_0D^\nu_* e^{-x^2}&=&\sum_{n=0}^{\infty}\dfrac{(-1)^n}{n!}\ _0D^\nu_* x^{2n}=\sum_{n=0}^{\infty}\dfrac{(-1)^n}{n!}\dfrac{\Gamma(2n+1)}{\Gamma(2n-\nu+1)}x^{2n-\nu}\\
&=&\dfrac{x^{-\nu}}{\Gamma(1-\nu)}\sum_{n=0}^{\infty}\dfrac{(1)_{2n}}{n!(1-\nu)_{2n}}(-x^2)^n\\
&=&\dfrac{x^{-\nu}}{\Gamma(1-\nu)}\sum_{n=0}^{\infty}\dfrac{(1)_{n}(\frac{1}{2})_n}{n!(\frac{1-\nu}{2})_n(\frac{2-\nu}{2})_n}(-x^2)^n\\
&=& \dfrac{x^{-\nu}}{\Gamma(1-\nu)}\hypergeom{2}{2}{\frac{1}{2},1}{\frac{1-\nu}{2},\frac{2-\nu}{2}}{-x^2}.
\end{eqnarray*}
\end{proof}

\subsection{Gray and Zhang fractional difference and their properties}
{ \fontfamily{times}\selectfont
 \noindent 

\begin{definition}\cite{gray} 
Let $\alpha$ and $\beta$ two complex numbers. $(\alpha)_{\beta}$ is
defined by:
\[
(\alpha)_{\beta}=\left\{\begin{array}{l}
         {\Gamma(\alpha+\beta)\over \Gamma(\alpha)}\;\;\text{ when }\; \alpha\; \text{and } \alpha+\beta\text{ are neither zero nor negative integers}\\
         1 ,\;\;\text{when} \;\alpha=\beta=0,\\
         0 , \;\; \text{when}\; \alpha=0, \beta \;\text{is not zero nor a negative integer,}\\
         undefined,\;\; \text{otherwise}.
       \end{array}\right.
\]
\end{definition}

\noindent Consider the $n$-fold summation of $f$ from $a$ to $t$, that is, let:
\[ \Ss{a}{t}{n} f(t)=\sum_{k_1=a}^t\sum_{k_2=a}^{k_1}\cdots \sum_{k_n=a}^{k_{n-1}}f(k_n)\]
where $t$, $k_i$ and $a$ are finite integers such that $a\leq k_i\leq k_{i-1}\leq t$.  Then by repeated interchanging of summation it is easily shown that 
\begin{equation}\label{discretecauchy}
\Ss{a}{t}{n} f(t)=\dfrac{1}{\Gamma(n)}\sum_{k=a}^{t}(t-k+1)_{n-1}f(k).
\end{equation}

\noindent Moreover, the summation in (\ref{discretecauchy}) is well defined for
 $n=\alpha$, $\alpha$ any complex number not zero or a negative
 integer. The definition can be extended to negative integers by
 noting that for $n$ a positive integer and $\alpha$ not zero or a
 negative integer
 \[
\nabla f(t)=f(t)-f(t-1)
 \] and
 \[
{\nabla^n\over
\Gamma(n+\alpha)}\sum_{k=a}^t(t-k+1)_{n+\alpha-1}f(k)={1\over
\Gamma(\alpha)}\sum_{k=a}^t(t-k+1)_{\alpha-1}f(k).
 \] Using this, we have the definition of the $\alpha$-fold summation of $f$ from $a$ to
 $t$.
 
\begin{definition}\cite{gray}
For $\alpha$ any complex number, and $f$ the function defined over
the integer set $\{a-n, a-n+1,\cdots,t\}$, the $\alpha$-order
summation over $\{a,a+1,\cdots,t\}$ is defined by:
\begin{equation}\label{e2.1a}
 \Ss{a}{t}{\alpha}
  f(t)={\nabla^n\over
\Gamma(n+\alpha)}\sum_{k=a}^t(t-k+1)_{n+\alpha-1}f(k)={1\over
\Gamma(\alpha)}\sum_{k=a}^t(t-k+1)_{\alpha-1}f(k)
\end{equation}
where $n=\max\{0,n_0\}$,  $n_0$ an integer such that
$0<Re(\alpha+n_0)\leq 1$.
\end{definition}

\begin{definition}\cite{gray}
For $\alpha$ any complex number, the $\alpha$th-order difference of
$f(t)$ over $\{a,a+1,\cdots,t\}$ is defined by
\[
\Nab{a}{t}{\alpha}f(t)=\Ss{a}{t}{-\alpha}f(t)
= {\nabla^n\over
\Gamma(n-\alpha)}\sum_{k=a}^t(t-k+1)_{n-\alpha-1}f(k)
\]
where $n=\max\{0,n_0\}$,  $n_0$ an integer such that
$0<Re(-\alpha+n_0)\leq 1$.
\end{definition}

\noindent{Taking care} that
\[
\sum_{k=a}^t(t-k+1)_{n+\alpha-1}f(k)=\sum_{k=0}^{t-a}(k+1)_{n+\alpha-1}f(t-k),
\]
the $\alpha$th-order difference of $f(t)$ over $\{a,a+1,\cdots,t\}$
can be defined as:
\begin{equation}
\label{e2}
\Nab{a}{t}{\alpha}f(t)={\nabla^n\over
\Gamma(n-\alpha)}\sum_{k=0}^{t-a}(k+1)_{n-\alpha-1}f(t-k).
\end{equation}

\begin{proposition}\cite{gray}\label{prop1}
For any complex number $\alpha$ and any nonnegative integer $p$ such
that $p-\alpha$ is not zero or a negative integer,
\[
\Nab{a}{t}{\alpha}f(t)={\nabla^p\over
\Gamma(p-\alpha)}\sum_{k=a}^t(t-k+1)_{p-\alpha-1}f(k).
\]
\end{proposition}

\begin{proposition}\cite{gray}
Let $\alpha$ and  $\beta$ two complex numbers. Depending on $\alpha$
and $\beta$, the following properties apply
\begin{enumerate}
\item If $\alpha$ and $\beta$ are zero or positive integers, then
\[
\Nab{a}{t}{\alpha}\Nab{a}{t}{\beta}f(t)=\Nab{a}{t}{\alpha+\beta}f(t).
\]
\item If $\alpha$ is any complex number and $\beta$ is not a
positive integer, then
\[
\Nab{a}{t}{\alpha}\Nab{a}{t}{\beta}f(t)=\Nab{a}{t}{\alpha+\beta}f(t).
\]
\item If $\alpha$ is not zero or a positive integer but $\beta$ is a
positive integer, then
\[
\Nab{a}{t}{\alpha}\Nab{a}{t}{\beta}f(t)=\Nab{a}{t}{\alpha+\beta}f(t)f(t)+{1\over
\Gamma(-\alpha)}\sum_{l=1}^{\beta}\sum_{j=a-l}^{a-1}(-1)^l\left(\begin{array}{c}
                                                               \beta \\
                                                               l
                                                             \end{array}\right)(t-l-j+1)_{-\alpha-1}f(j).
\]
\end{enumerate}
\end{proposition}

\begin{proposition}\cite{gray}
\begin{enumerate}
\item When $\alpha$ is not a negative integer, then
\[
\Nab{a}{t}{\alpha}\Nab{a}{t}{-\alpha}f(t)=f(t).
\]
\item For a constant $c$, we have
\[
\Nab{a}{t}{\alpha}\left(cf(t)+g(t)\right)=c
\Nab{a}{t}{\alpha}f(t)+\Nab{a}{t}{\alpha}g(t).
\]
\item  If $m$ is a nonnegative integer,
\[
\nabla^m(f(t)g(t))=\sum_{n=0}^m\binom{n}{m}\left[\nabla^{m-n}f(t-n)\right]\nabla^ng(t).
\]
\item if $\alpha$ is not a nonnegative integer,
\[
\Nab{a}{t}{\alpha}(f(t)g(t))=\sum_{n=0}^{t-a}\binom{n}{m}
\left[\Nab{a}{t-a}{\alpha-n}f(t-n)\right]\nabla^ng(t).
\]
\end{enumerate}
\end{proposition}

\begin{proposition}\cite{gray}
If $p+1$ is not zero or a negative integer, then
\begin{enumerate}
\item  when $p+1-\alpha$ is not zero or a negative integer
\[
\Nab{a+1}{t}{\alpha}(t-a)_p={t(t-a)_{p-\alpha}\over
(p+1)_{-\alpha}},
\]
\item when $p+1-\alpha$ is zero or a negative integer
\[
\Nab{a+1}{t}{\alpha}(t-a)_p=0.
\]
\end{enumerate}
\end{proposition}

\section{The fractional functions}
{ \fontfamily{times}\selectfont
 \noindent

\subsection{C-Hermite functions }
{ \fontfamily{times}\selectfont
 \noindent The classical Hermite polynomials are usually defined by the following Rodrigues' formula \cite[page 251]{KLS2010}

\begin{equation}\label{hermite0}
H_n(x)=(-1)^ne^{x^2}D^n\left(e^{-x^2}\right).
\end{equation}

\noindent  We generalize the Hermite polynomials by taking the Caputo fractional derivative instead of the integer-order derivative in \eqref{hermite0}. We then obtain the functions we call C-Hermite functions
\begin{equation}\label{Chermite}
H_{\nu}(x)=(-1)^ne^{x^2}\ _0D^\nu_*\left(e^{-x^2}\right).
\end{equation}

\begin{theorem}
Let $\nu>0$. Then the C-Hermite functions have the following hypergeometric representation
\begin{equation}\label{hyper-chermite}
H_{\nu}(x)=\dfrac{(-x)^{-\nu}e^{x^2}}{\Gamma(1-\nu)}\hypergeom{2}{2}{\frac{1}{2},1}{\frac{1-\nu}{2},\frac{2-\nu}{2}}{-x^2}.
\end{equation}
\end{theorem}

\begin{proof} Let $\nu>0$. We use the definition of $_0D^\nu_*$ and Equation \eqref{eq004} with $a=0$ to have:
\begin{eqnarray*}
H_{\nu}(x)&=& (-1)^{\nu}e^{x^2}\ _0D^\nu_*\left(e^{-x^2}\right)\\
&=& (-1)^\nu e^{x^2}  \sum_{k=0}^{\infty}\dfrac{x^{k-\nu}}{\Gamma(k+1-\nu)}D^k (e^{-x^2})\\
&=&(-1)^\nu\dfrac{x^{-\nu}e^{x^2}}{\Gamma(1-\nu)}\sum_{k=0}^{\infty}\dfrac{x^{k}}{(1-\nu)_k}[D^k e^{-x^2}](0)\\
&=&(-1)^\nu\dfrac{x^{-\nu}e^{x^2}}{\Gamma(1-\nu)}\sum_{k=0}^{\infty}\dfrac{x^{k}}{(1-\nu)_k}(-1)^k[e^{-x^2}H_k(x)](0)\\
&=&(-1)^\nu\dfrac{x^{-\nu}e^{x^2}}{\Gamma(1-\nu)}\sum_{k=0}^{\infty}\dfrac{x^{2k}}{(1-\nu)_{2k}}H_{2k}(0)\\
&=&(-1)^\nu\dfrac{x^{-\nu}e^{x^2}}{\Gamma(1-\nu)}\sum_{k=0}^{\infty}\dfrac{(2k)!}{k!(1-\nu)_{2k}}(-x^2)^k\\
&=&(-1)^\nu\dfrac{x^{-\nu}e^{x^2}}{\Gamma(1-\nu)}\hypergeom{2}{2}{\frac{1}{2},1}{\frac{1-\nu}{2},\frac{2-\nu}{2}}{-x^2}.
\end{eqnarray*}
Note that the result could be obtained directly using Theorem \ref{theorem00}.
\end{proof}

\subsection{C-Laguerre functions }
{ \fontfamily{times}\selectfont
 \noindent The classical Laguerre polynomials are defined by the following Rodrigues' formula \cite[page 242]{KLS2010}
\begin{equation}
L_n^{(\alpha)}(x)=\dfrac{1}{n!}\dfrac{e^x}{x^\alpha}D^n\left(e^{-x}x^{n+\alpha}\right).
\end{equation}

 The C-Laguerre functions (see \cite{ishteva}) are defined by
 \begin{equation}\label{c-laguerre}
 L_{\nu}^{(\alpha)}(x)=\dfrac{1}{\Gamma(\nu+1)}x^{-\alpha}e^{x}D_{*}^{\nu}(e^{-x}x^{\nu+\alpha}),
 \end{equation}
 where $n\in\N$, $\Re(\alpha)>0$, \; $x,\alpha\in\R$ and $n-1<\alpha<n$.
 The following relations are valid.

 \begin{theorem}\cite{ishteva}\label{theolag}
 Let $n\in\N$, $\alpha\in\C$, $\Re(\alpha)>0$, $x,\nu\in\R$,\, $n-1<\nu<n$. Then the C-Laguerre functions can be represented by means of the confluent hypergeometric functions as
 \begin{equation}
 L_{\nu}^{(\alpha)}(x)=\dfrac{\Gamma(\nu+\alpha+1)}{\Gamma(\nu+1)\Gamma(\alpha+1)}\hypergeom{1}{1}{-\nu}{\alpha+1}{x}.
 \end{equation}
 \end{theorem}

 \begin{proposition}\cite{ishteva}
 Let the conditions of \color{black}{Theorem \ref{theolag}}\color{black} be fulfilled. Then\newline

     \hspace*{1cm} (1)  $\lim\limits_{\nu\to n}L_{\nu}^{(\alpha)}(x)=L_{n}^{(\alpha)}(x)$\\

     \hspace*{1cm} (2) $\dfrac{d}{dx}L_{\nu}^{(\alpha)}(x)=-L_{\nu-1}^{(\alpha+1)}(x).$
 \end{proposition}

 \begin{theorem}
 The C-Laguerre functions satisfy the differential equation
\begin{equation}
xy'' + (\alpha+1 - x)y' + \nu y = 0.
\end{equation}
 \end{theorem}

 \textcolor{black}{\begin{remark}
 It happens that the $C$-Laguerre functions satisfy the  Kummer differential equation \cite[Page 83]{Nico}.
 \end{remark}}

\subsection{C-Jacobi functions }
{ \fontfamily{times}\selectfont
 \noindent 
 The classical Jacobi polynomials are usually defined by the following Rodrigues' formula \cite[page 251]{KLS2010}
 \begin{equation}\label{e9}
 (1-x)^{\alpha}(1+x)^{\beta}P_n^{(\alpha,\beta)}(x)=\dfrac{(-1)^n}{2^nn!}D^n\left[ (1-x)^{n+\alpha}(1+x)^{n+\beta}\right].
 \end{equation}

\noindent We generalize the Jacobi polynomials by taking the Caputo fractional derivative (\ref{e2}), intead of the integer-order
derivative formula (\ref{e9}) and substituting $n$ by $\nu$ and $n!$ by $\Gamma(\nu+1)$. We obtain the functions we call C-Jacobi functions:
\begin{equation}\label{e10}
(1-x)^\alpha (1+x)^\beta P_\nu^{(\alpha,\beta)}(x)={(-1)^\nu\over 2^\nu \Gamma(\nu+1)}\; \\ _{-1} D^{\nu}_*\left[(1-x)^{\nu+\alpha} (1+x)^{\nu+\beta}\right],
\end{equation}
where $n\in \N,\;\;\alpha,\;\beta\in\R, \;\alpha>0,\;\beta>0,\;\;x,\nu\in \R,\;\;n-1<\nu<n.$

We have the following result.

\begin{theorem}\label{thm1}
Let $n\in \N, \; \alpha, \;\beta \in \R, \;\alpha> -1,\; \beta> -1,\; x, \;\nu\in \R, \; n-1<\nu<n.$ The C-Jacobi functions have
the following hypergeometric representations
\begin{eqnarray}
P_\nu^{(\alpha,\beta)}(x)&=&\frac{(-1)^\nu \Gamma(\nu+\beta+1)}{  \Gamma(\nu+1)\Gamma(\beta+1)}\left(\frac{1-x}{2}\right)^{\nu}\hypergeom{2}{1}{-\nu,-\nu-\alpha}{\beta+1}{\frac{1+x}{1-x}}\nonumber\\
&=&\frac{(-1)^\nu \Gamma(\nu+\beta+1)}{\Gamma(\nu+1)\Gamma(\beta+1)}
\hypergeom{2}{1}{-\nu,\nu+\alpha+\beta+1}{\beta+1}{\frac{x+1}{ 2}}.\label{cjacobihyper}
\end{eqnarray}
\end{theorem}

\begin{proof}
From \eqref{e7}, we write
{\small \begin{eqnarray*}
_{-1} D^{\nu}_*\left[(1-x)^{\nu+\alpha} (1+x)^{\nu+\beta}\right]&=&\sum_{k=0}^{\infty}\binom{\nu}{k}\; _{-1} D^{\nu-k}[(x+1)^{\nu+\beta}] \left((1-x)^{\nu+\alpha}\right)^{(k)}\\
&&\hspace*{0cm}-\sum_{k=0}^{n-1}\dfrac{(t+1)^{k-\nu}}{\Gamma(k+1-\nu)}\left[ (1-x)^{\nu+\alpha}(1+x)^{\nu+\beta}\right]^{(k)}(-1)\\
&=&\sum_{k=0}^{\infty}\binom{\nu}{k}\; _{-1} D^{\nu-k}[(x+1)^{\nu+\beta}] \left((1-x)^{\nu+\alpha}\right)^{(k)}.
\end{eqnarray*}}
From the relations
\begin{eqnarray*}
 _{-1} D^{\nu-k}[(x+1)^{\nu+\beta}]&=&\dfrac{\Gamma(\nu+\beta+1)}{\Gamma(\beta+k+1)}(x+1)^{\beta+k},\\
\left ((1-x)^{\nu+\alpha}\right)^{(k)}&=&(-\nu-\alpha)_k(1-x)^{\nu+\alpha-k},
\end{eqnarray*}
it follows that
\begin{eqnarray*}
&&_{-1} D^{\nu}_*\left[(1-x)^{\nu+\alpha} (1+x)^{\nu+\beta}\right]\\
&&\hspace*{0.6cm}=\dfrac{\Gamma(\nu+\beta+1)}{\Gamma(\beta+1)}(1-x)^{\nu+\alpha}(1+x)^{\beta}\sum_{k=0}^{\infty}\dfrac{(-\nu)_k(-\nu-\alpha)_k}{k!(\beta+1)_k}\left(\dfrac{1+x}{1-x}\right)^k\\
&&\hspace*{0.6cm}=\dfrac{\Gamma(\nu+\beta+1)}{\Gamma(\beta+1)}(1-x)^{\nu+\alpha}(1+x)^{\beta}\hypergeom{2}{1}{-\nu,-\nu-\alpha}{\beta+1}{\frac{1+x}{1-x}}.
\end{eqnarray*}
Hence we have
\begin{equation*}
P_\nu^{(\alpha,\beta)}(x)=\frac{(-1)^\nu \Gamma(\nu+\beta+1)}{  \Gamma(\nu+1)\Gamma(\beta+1)}\left(\frac{1-x}{2}\right)^{\nu}\hypergeom{2}{1}{-\nu,-\nu-\alpha}{\beta+1}{\frac{1+x}{1-x}}.
\end{equation*}
Next, doing the change of the variable $x=2z-1$ we have $z=\dfrac{x+1}{2}$ and then
\begin{eqnarray*}
P_\nu^{(\alpha,\beta)}(x)&=& {(-1)^\nu \Gamma(\nu+\beta+1)\over  \Gamma(\nu+1)\Gamma(\beta+1)}
(1-z)^{\nu}\hypergeom{2}{1}{-\nu,-\nu-\alpha}{\beta+1}{\frac{z}{z-1}}.\label{e15}
\end{eqnarray*}
Using the property (\ref{e3}) of the hypergeometric function ${_2F_1}$, it follows that
\begin{eqnarray*}
P_\nu^{(\alpha,\beta)}(x)&=& \frac{(-1)^\nu \Gamma(\nu+\beta+1)}{\Gamma(\nu+1)\Gamma(\beta+1)}
\hypergeom{2}{1}{-\nu,\nu+\alpha+\beta+1}{\beta+1}{z}\\
&=& \frac{(-1)^\nu \Gamma(\nu+\beta+1)}{\Gamma(\nu+1)\Gamma(\beta+1)}
\hypergeom{2}{1}{-\nu,\nu+\alpha+\beta+1}{\beta+1}{\frac{x+1}{2}}.
\end{eqnarray*}
The proposition is then proved.
\end{proof}

\begin{remark}
It can be seen that when $\nu$ tends to a
nonnegati\color{black}{v}\color{black}e integer $n$, the C-Jacobi
function $P_{\nu}^{(\alpha,\beta)}(x)$ becomes the classical Jacobi
polynomial $P_{n}^{(\alpha,\beta)}(x)$.
\end{remark}

\begin{theorem}\label{thm2a}
Let $n\in \N, \; \alpha, \;\beta \in \R, \;\alpha>-1,\; \beta>0,\; x, \;\nu\in \R, \; n-1<\nu<n.$
The C-Jacobi functions are solutions of the second order differential equation
{\begin{eqnarray}\label{e17a}
&&(1-x^2)\left(P_\nu^{(\alpha,\beta)}\right)''(x)\nonumber\\
&&\quad+\left(\beta-\alpha-(\alpha+\beta+2)x\right)\left(P_\nu^{(\alpha,\beta)}\right)'(x)+\nu(\nu+\alpha+\beta+1)P_\nu^{(\alpha,\beta)}(x)=0.
\end{eqnarray}}
\end{theorem}

\begin{proof}
Since  the functions $P_\nu^{(\alpha,\beta)}(x)$ have the hypergeometric representation \eqref{cjacobihyper}, the differential equation follows from \eqref{equadiff}.
\end{proof}

\begin{theorem} The Jacobi functions fulfil the following relation
{\begin{equation}
\dfrac{d^k}{dx^k}P_\nu^{(\alpha,\beta)}(x)=\dfrac{(\nu+\alpha+\beta)_k}{2^k}P_{\nu-k}^{(\alpha+k,\beta+k)}(x),\quad k\in\N.
\end{equation}}
\end{theorem}

\begin{proof}
The proof follows from  relation \cite[page 41]{wilhelm}
\[ \dfrac{d^k}{dx^k}\hypergeom{2}{1}{a,b}{c}{x}=\frac{(a)_k(b)_k}{(c)_k}\hypergeom{2}{1}{a+k,b+k}{c+k}{x}.\]
\end{proof}

\noindent In the following subsections, we list some particular cases of the C-Jacopi functions.

\subsubsection{The C-Gegenbauer functions}
{ \fontfamily{times}\selectfont
 \noindent 

\noindent The classical Gegenbauer polynomials are Jacobi polynomials for $\alpha=\beta=\lambda-\frac{1}{2}$. They can be defined by the Rodrigues'  formula
\begin{equation}
(1-x^2)^{\lambda-\frac{1}{2}}C_{n}^{\lambda}(x)=\dfrac{(-1)^n(2\lambda)}{(\lambda+\frac{1}{2})_n2^nn!} D^n\left[(1-x^2)^{\lambda+n-\frac{1}{2}}\right].
\end{equation}

\noindent The C-Gegenbauer functions are defined by
\begin{equation}\label{e19}
(1-x^2)^{\lambda-\frac{1}{2}}C_\nu^{(\lambda)}(x)={(-1)^\nu\over 2^\nu \Gamma(\nu+1)}\ _1 D^{\nu}_*\left[(1-x^2)^{\nu+\lambda-\frac{1}{ 2}}\right],
\end{equation}
$n\in \N, \; \lambda,\; x, \;\nu\in \R,\; \lambda\geq \frac{1}{2},\; \; n-1<\nu<n.$

\begin{theorem}\label{thm4}
Let $n\in \N, \; \lambda,\; x, \;\nu\in \R,\; \lambda\geq {1\over 2},\; \; n-1<\nu<n$. The C-Gegenbauer functions (\ref{e19}) have the following hypergeometric representation
\[C_\nu^{(\lambda)}(x)={(-1)^\nu \Gamma(\nu+\lambda+{1\over 2})\over  \Gamma(\nu+1)\Gamma(\lambda+{1\over 2})}
\hypergeom{2}{1}{-\nu,\nu+2\lambda}{\lambda+\frac{1}{2}}{\frac{x+1}{2}}.\]
\end{theorem}

\begin{proof}
The proof is similar to the one of Theorem \ref{thm1}.
\end{proof}

\begin{remark}
It can be seen that for $\nu$ approaching a natural number, The C-Gegenbauer functions become the classical Gegenbauer polynomials.
\end{remark}

\begin{proposition}
Let $n\in \N, \; \lambda,\; x, \;\nu\in \R,\; \lambda\geq \frac{1}{2},\; \; n-1<\nu<n.$
The C-Gegenbauer functions are solutions of the second-order differential equation
\begin{equation}\label{e19a}
(1-x^2)\left(C_\nu^{(\lambda)}\right)''(x)-(2\lambda+1)x\left(C_\nu^{(\lambda)}\right)'(x)
+\nu(\nu+2\lambda)C_\nu^{(\lambda)}(x)=0.
\end{equation}
\end{proposition}

\begin{proof}
The proof follows from \eqref{e17a} by taking $\alpha=\beta=\lambda-\dfrac{1}{2}$.
\end{proof}

\subsubsection{The C-Chebyshev functions}
{ \fontfamily{times}\selectfont
 \noindent The Chebyshev polynomials of the first kind $T_n(x)$ and $U_n(x)$ are Jacobi polynomials for $\alpha=\beta=-\frac{1}{2}$ and $\alpha=\beta=\frac{1}{2}$ respectively. The have the Rodrigues' representations \cite[page 225]{KLS2010}
\begin{eqnarray}
(1-x^2)^{-\frac{1}{2}}   T_n(x)&=& \dfrac{(-1)^n}{(\frac{1}{2})_n2^n} D^n\left[(1-x^2)^{n-\frac{1}{2}}\right]  \\
(1-x^2)^{\frac{1}{2}}   U_n(x)&=& \dfrac{(n+1)(-1)^n}{(\frac{3}{2})_n2^n} D^n\left[(1-x^2)^{n+\frac{1}{2}}\right],
\end{eqnarray}

\noindent The C-Chebyshev functions are defined by
\begin{eqnarray}\label{e18}
(1-x^2)^{-\frac{1}{2}} T_n(x)&=&{(-1)^\nu\over 2^\nu \Gamma(\nu+1)}\;  _1 D^{\nu}_*\left[(1-x)^{\nu-\frac{1}{2}} (1+x)^{\nu-\frac{1}{2}}\right],\\
(1-x^2)^{\frac{1}{2}} U_n(x)&=&{(-1)^\nu\over 2^\nu \Gamma(\nu+1)}\;  _1 D^{\nu}_*\left[(1-x)^{\nu+\frac{1}{2}} (1+x)^{\nu+\frac{1}{2}}\right],
\end{eqnarray}

\begin{theorem}\label{thm3}
Let $n\in \N, \; \,\; x, \;\nu\in \R, \; n-1<\nu<n.$ The C-Chebyshev functions have the following hypergeometric representations
\begin{eqnarray}
T_{\nu}(x)&=&  \frac{(-1)^\nu \Gamma(\nu+\frac{1}{2})}{\Gamma(\nu+1)\Gamma(\frac{1}{2})}
\hypergeom{2}{1}{-\nu,\nu}{\frac{1}{2}}{\frac{x+1}{ 2}}.     \\
U_\nu(x)&=&\frac{(-1)^\nu \Gamma(\nu+\frac{3}{2})}{ \Gamma(\nu+1)\Gamma(\frac{3}{2})}
\hypergeom{2}{1}{-\nu,\nu+2}{\frac{3}{2}}{\frac{x+1}{2}}.
\end{eqnarray}
\end{theorem}

\begin{proof}
The proof is similar to the one of Theorem \ref{thm1}.
\end{proof}

\subsubsection{The Legendre C-functions}
{ \fontfamily{times}\selectfont
 \noindent The Legendre polynomials are Jacobi polynomials for $\alpha=\beta=0$. The have can be defined by the Rodrigues' formula
\begin{equation}
P_n(x)=\dfrac{(-1)^n}{2^nn!}D^n\left[(1-x^2)^n\right]
\end{equation}
The C-Legendre functions are defined by
\begin{equation}\label{e17}
P_\nu(x)={(-1)^\nu\over 2^\nu \Gamma(\nu+1)}\  _1D^{\nu}_*\left[(1-x^2)^{\nu} \right].
\end{equation}

\begin{theorem}\label{thm2}
Let $n\in \N, \; \,\; x, \;\nu\in \R, \; n-1<\nu<n.$ The C-Legendre functions (\ref{e17}) have the hypergeometric representation
\[P_\nu(x)=(-1)^\nu
\hypergeom{2}{1}{-\nu,\nu+1}{1}{\frac{x+1}{2}}.\]
\end{theorem}
\begin{remark}
It can be seen that for $\nu$ tends to a nonnegative integer $n$, the C-Legendre functions $P_\nu(x)$ become the classical Legendre polynomials $P_n(x)$. Note that the C-Legendre functions can also be written as
\begin{equation}
P_\nu(x)=\hypergeom{2}{1}{-\nu,\nu+1}{1}{\frac{1-x}{2}}.
\end{equation}
It is not difficult to the see that
\[P_\nu(x)=\hypergeom{2}{1}{\nu+1,-\nu}{1}{\frac{1-x}{2}}=P_{-\nu-1}(x).\]
\end{remark}

\begin{proposition}\label{pr1}
Let $n\in \N, \; \,\; x, \;\nu\in \R, \; n-1<\nu<n.$ The C-Legendre functions are solutions of the second-order differential equation
\begin{equation}\label{e17a}
(1-x^2)P_\nu''(x)-2xP_\nu'(x)
+\nu(\nu+1)P_\nu(x)=0.
\end{equation}
\end{proposition}

\begin{remark}
The C-Legendre functions happen to be the Legendre functions defined in \cite[page 195]{Nico} and so our investigation leads to a Rodrigues representation  of the Legendre functions using a derivative of a fractional order.
\end{remark}

In the sections, using the Gray and Zhang  factional
difference \cite{gray}, we define the Fractional Charlier,  Fractional Meixner,
    the Fractional Krawtchouk and the Fractional Hahn functions and provide several properties of these functions.

\subsection{Fractional Charlier functions }
{ \fontfamily{times}\selectfont
 \noindent \begin{definition}
Let $a\in \mathbb{C}^*$, $\mu\in \mathbb{C}$ and $x\in
\{0,1,\cdots\}.$ We define the fractional Charlier function
$C_{\mu}(x;a)$ using the Rodrigues type formula as
\begin{equation}\label{charlier1}
C_\mu(x;a)={x!\over
a^x}\Nab{0}{x}{\mu}\left[{a^x\over
x!}\right].
\end{equation}
\end{definition}

\begin{proposition}
The fractional Charlier functions $C_{\mu}(x;a)$ have the following hypergeometric representation
\begin{equation}\label{charlier2}
C_{\mu}(x;a)=\hypergeom{2}{0}{-\mu,-x}{-}{-\frac{1}{a}}.
\end{equation}
\end{proposition}

\begin{proof}
Applying Proposition \ref{prop1} where $p$ is chosen to be 0, we have:
\[
C_\mu(x;a)={x!\over a^x}{1\over
\Gamma(-\mu)}\sum_{k=0}^x(x-k+1)_{-\mu-1}{a^k\over k!}.
\]
If we put $j=x-k$ then the previous relation becomes:
\begin{eqnarray}
C_\mu(x;a)&=&{x!\over a^x}{1\over
\Gamma(-\mu)}\sum_{j=x}^0(j+1)_{-\mu-1}{a^{x-j}\over
(x-j)!}.\nonumber\\
&=&{1\over \Gamma(-\mu)}\sum_{j=0}^x(j+1)_{-\mu-1}x(x-1)\cdots
(x-j+1) a^{-j}
\nonumber\\
&=&\sum_{j=0}^x {\Gamma(j-\mu)\over
\Gamma(-\mu)\Gamma(j+1)}(-x)_j\left(-{1\over a}\right)^{j}
\nonumber\\
&=&\sum_{j=0}^x {(-\mu)_j(-x)_j\over j!}\left(-{1\over
a}\right)^{j}. \nonumber
\end{eqnarray}
{Since} $x\in \mathbb{N}$ {then}  $(-x)_j=0$
for $j>x$ {and}
\[
C_\mu(x;a)=\sum_{j=0}^{+\infty} {(-\mu)_j(-x)_j\over
j!}\left(-{1\over
a}\right)^{j}=\hypergeom{2}{0}{-\mu,-x}{-}{-\frac{1}{a}}.
\]
\end{proof}

 \subsection{Fractional Meixner functions }
{ \fontfamily{times}\selectfont
 \noindent \begin{definition}
Let $c\in \mathbb{C}^*$, $\beta\in \mathbb{C},$  $\mu\in \mathbb{C}$
and $x\in \{0,1,\cdots\}.$ We define the fractional Meixner function
$M_{\mu}(x;\beta,c)$ using the Rodrigues Type formula as
\begin{equation}\label{Meixner1}
M_\mu(x;\beta,c)={x!\over  c^x(\beta)_x
}\Nab{0}{x}{\mu}\left[
{(\beta+\mu)_xc^x\over x!}\right].
\end{equation}
\end{definition}

\begin{proposition}
The fractional Meixner function $M_\mu(x;\beta,c)$ satisfies the
following sum:
\begin{equation}\label{Meixner2}
M_\mu(x;\beta,c)= {\Gamma(\beta)\Gamma(\beta+\mu+x)\over
\Gamma(\beta+x)\Gamma(\beta+\mu)}\sum_{j=0}^x{(-\mu)_j(-x)_j\over
j!(-\beta-\mu-x+1)_j}\left({1\over c}\right)^{j}.
\end{equation}
\end{proposition}

\begin{proof}  
{From the definitions of the fractional difference and the fractional Meixner function, we have}
\begin{eqnarray}
M_{\mu}(x;\beta,c)&=&{x!\over
(\beta)_xc^x}\Nab{0}{x}{\mu}\left[
{(\beta+\mu)_xc^x\over x!}\right]\nonumber\\
&=&{x!\over (\beta)_xc^x\Gamma(-\mu)}\sum_{k=0}^x(x-k+1)_{-\mu-1}
{(\beta+\mu)_kc^k\over k!}\nonumber\\
&=&{x!\Gamma(\beta)\over \Gamma(\beta+x)c^x}{1\over
\Gamma(-\mu)}\sum_{k=0}^x{\Gamma(x-k-\mu)\Gamma(\beta+\mu+k)\over
\Gamma(x-k+1)\Gamma(\beta+\mu)}{c^k\over k!}.\nonumber
\end{eqnarray}
If we put $j=x-k$ then the previous relation becomes:
\begin{eqnarray}
M_{\mu}(x;\beta,c)&=&{x!\Gamma(\beta)\over
\Gamma(\beta+x)c^x}{1\over
\Gamma(-\mu)}\sum_{j=x}^0{\Gamma(j-\mu)\Gamma(\beta+\mu+x-j)\over
\Gamma(j+1)\Gamma(\beta+\mu)}{c^{x-j}\over (x-j)!}\nonumber\\
&=&{x!\Gamma(\beta)\over \Gamma(\beta+x)c^x}{1\over
\Gamma(-\mu)}\sum_{j=0}^x{\Gamma(j-\mu)\Gamma(\beta+\mu+x-j)\over
\Gamma(j+1)\Gamma(\beta+\mu)}{c^{x-j}\over (x-j)!}\nonumber\\
&=&{\Gamma(\beta)\over
\Gamma(\beta+x)}\sum_{j=0}^x{(-\mu)_j(-x)_j\over
j!}{\Gamma(\beta+\mu+x-j)\over \Gamma(\beta+\mu)}\left(-{1\over
c}\right)^{j}\nonumber\\
 &=&{\Gamma(\beta)\Gamma(\beta+\mu+x)\over
\Gamma(\beta+x)\Gamma(\beta+\mu)}\sum_{j=0}^x{(-\mu)_j(-x)_j\over
j!}{\Gamma(\beta+\mu+x-j)\over \Gamma(\beta+\mu+x)}\left(-{1\over
c}\right)^{j}\nonumber\\
&=&{\Gamma(\beta)\Gamma(\beta+\mu+x)\over
\Gamma(\beta+x)\Gamma(\beta+\mu)}\sum_{j=0}^x{(-\mu)_j(-x)_j\over
j!(-\beta-\mu-x+1)_j}\left({1\over c}\right)^{j}.\nonumber
\end{eqnarray}
{Since} $x\in \mathbb{N}$ {then} $(-x)_j=0$
for $j>x$ {and}
\begin{eqnarray}
M_{\mu}(x;\beta,c)&=&{\Gamma(\beta)\Gamma(\beta+\mu+x)\over
\Gamma(\beta+x)\Gamma(\beta+\mu)}\sum_{j=0}^{\infty}{(-\mu)_j(-x)_j\over
j!(-\beta-\mu-x+1)_j}\left({1\over c}\right)^{j}\nonumber\\
&=&{\Gamma(\beta)\Gamma(\beta+\mu+x)\over
\Gamma(\beta+x)\Gamma(\beta+\mu)}\hypergeom{2}{1}{-\mu,-x}{-\beta-\mu-x+1}{{1\over
                           c}}.\nonumber
\end{eqnarray}
\end{proof}

\subsection{Fractional Krawtchouk functions }
{ \fontfamily{times}\selectfont
 \noindent \begin{definition}
Let $p\in \mathbb{C}^*$, $N\in \mathbb{N},$  $\mu\in \mathbb{C}$ and
$x\in \{0,1,\cdots\}.$ We define the fractional Krawtchouk function
$K_\mu(x;p,N)$ using the Rodrigues type formula as
\begin{equation}\label{krawtchouk1}
K_\mu(x;p,N)={1\over \binom{N}{x} \left({p\over
                           1-p}\right)^x
}\Nab{0}{x}{\mu}\left[\binom{N-\mu}{ x} \left({p\over
                           1-p}\right)^x\right].
\end{equation}
\end{definition}

\begin{proposition}
The fractional Krawtchouk function $K_\mu(x;p,N)$ satisfies the
following sum:
\begin{equation}\label{krawtchouk2}
K_\mu(x;p,N)={\Gamma(N-x+1)\Gamma(N-\mu+1)\over
\Gamma(N+1)\Gamma(N-\mu-x+1)}\sum_{j=0}^x{(-\mu)_j(-x)_j\over
                           j!(N-\mu-x+1)_j}\left(1-{1\over
                           p}\right)^{j}.
\end{equation}
\end{proposition}

\begin{proof}
{From the definitions of the fractional difference and the fractional Krawtchouk function, we have}
\begin{eqnarray}
K_\mu(x;p,N)&=&{1\over \binom{N}{x} \left({p\over
                           1-p}\right)^x
}\Nab{0}{x}{\mu}\left[\binom{N-\mu}{x} \left({p\over
                           1-p}\right)^x\right]\nonumber\\
&=&{1\over \binom{N}{x} \left({p\over
                           1-p}\right)^x
}{1\over \Gamma(-\mu)}\sum_{j=0}^x(x-k+1)_{-\mu-1}
\binom{N-\mu}{k} \left({p\over
                           1-p}\right)^k\nonumber\\
&=&{\Gamma(x+1)\Gamma(N-x+1)\over \Gamma(N+1)\left({p\over
                           1-p}\right)^x}{1\over
                           \Gamma(-\mu)}\sum_{k=0}^x{\Gamma(x-\mu-k)\Gamma(N-\mu+1)\over
                           \Gamma(x-k+1)\Gamma(k+1)\Gamma(N-\mu-k+1)}\left({p\over
                           1-p}\right)^k.\nonumber
\end{eqnarray}
If we put $j=x-k$ then the previous relation becomes:
\begin{eqnarray}
K_\mu(x;p,N)&=&{\Gamma(x+1)\Gamma(N-x+1)\over
\Gamma(N+1)\left({p\over
                           1-p}\right)^x}{1\over
                           \Gamma(-\mu)}\sum_{j=x}^0{\Gamma(j-\mu)\Gamma(N-\mu+1)\over
                           \Gamma(j+1)\Gamma(x-j+1)\Gamma(N-\mu-x+j+1)}\left({p\over
                           1-p}\right)^{x-j}\nonumber\\
                           &=&{\Gamma(N-x+1)\over
\Gamma(N+1)}\sum_{j=0}^x{(-\mu)_j(-x)_j\Gamma(N-\mu+1)\over
                           j!\Gamma(N-\mu-x+j+1)}(-1)^j\left({p\over
                           1-p}\right)^{-j}\nonumber\\
                           &=&{\Gamma(N-x+1)\Gamma(N-\mu+1)\over
\Gamma(N+1)\Gamma(N-\mu-x+1)}\sum_{j=0}^x{(-\mu)_j(-x)_j\Gamma(N-\mu-x+1)\over
                           j!\Gamma(N-\mu-x+j+1)}(-1)^j\left({p\over
                           1-p}\right)^{-j}\nonumber\\
                           &=&{\Gamma(N-x+1)\Gamma(N-\mu+1)\over
\Gamma(N+1)\Gamma(N-\mu-x+1)}\sum_{j=0}^x{(-\mu)_j(-x)_j\over
                           j!(N-\mu-x+1)_j}\left(1-{1\over
                           p}\right)^{j}.\nonumber
 \end{eqnarray}
{Since} $x\in \mathbb{N}$ {then} $(-x)_j=0$
for $j>x$ {and}
\begin{eqnarray}
K_\mu(x;p,N)&=&{\Gamma(N-x+1)\Gamma(N-\mu+1)\over
\Gamma(N+1)\Gamma(N-\mu-x+1)}\sum_{j=0}^{+\infty}{(-\mu)_j(-x)_j\over
                           j!(N-\mu-x+1)_j}\left(1-{1\over
                           p}\right)^{j}\nonumber\\
&=&{\Gamma(N-x+1)\Gamma(N-\mu+1)\over
\Gamma(N+1)\Gamma(N-\mu-x+1)}\hypergeom{2}{1}{-\mu,-x}{N-\mu-x+1}{1-{1\over
                           p}}.\nonumber
\end{eqnarray}
\end{proof}

\subsection{Fractional Hahn functions }
{ \fontfamily{times}\selectfont
 \noindent \begin{definition}
Let $\alpha,\;\beta\in \mathbb{C}$, $N\in \mathbb{N},$  $\mu\in
\mathbb{C}$ and $x\in \{0,1,\cdots\}.$ We define the fractional Hahn
function $Q_\mu(x;\alpha, \beta,N)$ using the Rodrigues type formula
as
\begin{equation}\label{hahn1}
Q_\mu(x;\alpha, \beta,N)={(-1)^{\mu}(\beta+1)_{\mu}\over (-N)_{\mu}
\binom{ \alpha+x}{x}\binom{\beta+N-x}{ N-x}
}\Nab{0}{x}{\mu}\left[\binom{\alpha+\mu+x}{x}\binom{\beta+N-x}{
                             N-\mu-x}\right].
\end{equation}
\end{definition}

\begin{proposition}
The fractional Hahn function $Q_\mu(x;\alpha, \beta,N)$ satisfies
the following sum:
\begin{eqnarray}
Q_\mu(x;\alpha,
\beta,N)&=&{(-1)^\mu\Gamma(\beta+\mu+1)\Gamma(-N)\Gamma(\alpha+1)\Gamma(N-x)\Gamma(\alpha+\mu+x+1)\over
\Gamma(-N+\mu)\Gamma(\alpha+x+1)\Gamma(\alpha+\mu+1)\Gamma(N-\mu-x)\Gamma(\beta-\mu+1)}\nonumber\\
&&\times\hypergeom{3}{2}{-\mu,-x,\beta+N-x+1}{N-\mu-x,-\alpha-\mu-x}{1}.\nonumber
\end{eqnarray}
\end{proposition}

\begin{proof}
\begin{eqnarray}
Q_\mu(x;\alpha, \beta,N)&=&{(-1)^{\mu}(\beta+1)_{\mu}\over (-N)_{\mu}
\binom{ \alpha+x}{x}\binom{\beta+N-x}{ N-x}
}\Nab{0}{x}{\mu}\left[\binom{\alpha+\mu+x}{x}\binom{\beta+N-x}{
                             N-\mu-x}\right]\nonumber\\
&=&{(-1)^{\mu}(\beta+1)_{\mu}\over (-N)_{\mu}
\binom{ \alpha+x}{x}\binom{\beta+N-x}{ N-x}
}{1\over \Gamma(-\mu)}\nonumber\\
&&\times\sum_{k=0}^x (x-k+1)_{-\mu-1}\left[\binom{\alpha+\mu+k}{k}\binom{\beta+N-k}{
                             N-\mu-k}\right]
\nonumber\\
&=&{(-1)^\mu\Gamma(\beta+\mu+1)\Gamma(-N)\Gamma(x+1)\Gamma(\alpha+1)\Gamma(N-x)\over
\Gamma(-N+\mu)\Gamma(\alpha+x+1)\Gamma(\beta+N-x+1)\Gamma(-\mu)}\nonumber\\
&&\times\sum_{k=0}^x{\Gamma(x-k-\mu)\Gamma(\alpha+\mu+k+1)\Gamma(\beta+N-k+1)\over
\Gamma(x-k+1)\Gamma(k+1)\Gamma(\alpha+\mu+1)\Gamma(N-\mu-k+1)\Gamma(\beta-\mu+1)}.\nonumber
\end{eqnarray}
If we put $j=x-k$ then the previous relation becomes:
\begin{eqnarray}
Q_\mu(x;\alpha,
\beta,N)&=&{(-1)^\mu\Gamma(\beta+\mu+1)\Gamma(-N)\Gamma(x+1)\Gamma(\alpha+1)\Gamma(N-x)\over
\Gamma(-N+\mu)\Gamma(\alpha+x+1)\Gamma(\beta+N-x+1)\Gamma(-\mu)}\nonumber\\
&&\times\sum_{j=x}^0{\Gamma(j-\mu)\Gamma(\alpha+\mu+x-j+1)\Gamma(\beta+N-x+j+1)\over
\Gamma(j+1)\Gamma(x-j+1)\Gamma(\alpha+\mu+1)\Gamma(N-\mu-x+j+1)\Gamma(\beta-\mu+1)}\nonumber\\
                           &=&{(-1)^\mu\Gamma(\beta+\mu+1)\Gamma(-N)\Gamma(x+1)\Gamma(\alpha+1)\Gamma(N-x)\over
\Gamma(-N+\mu)\Gamma(\alpha+x+1)\Gamma(\beta+N-x+1)\Gamma(-\mu)}\nonumber\\
&&\times\sum_{j=0}^x{\Gamma(j-\mu)\Gamma(\alpha+\mu+x-j+1)\Gamma(\beta+N-x+j+1)\over
\Gamma(j+1)\Gamma(x-j+1)\Gamma(\alpha+\mu+1)\Gamma(N-\mu-x+j+1)\Gamma(\beta-\mu+1)}\nonumber\\
&=&{(-1)^\mu\Gamma(\beta+\mu+1)\Gamma(-N)\Gamma(\alpha+1)\Gamma(N-x)\Gamma(\alpha+\mu+x+1)\over
\Gamma(-N+\mu)\Gamma(\alpha+x+1)\Gamma(\alpha+\mu+1)\Gamma(N-\mu-x)\Gamma(\beta-\mu+1)}\nonumber\\
&&\times\sum_{j=0}^x{(-\mu)_j(-x)_j(\beta+N-x+1)_j\over
j!(N-\mu-x)_j(-\alpha-\mu-x)_j}.\nonumber
 \end{eqnarray}
 {Since} $x\in \mathbb{N}$ {then} $(-x)_j=0$
for $j>x$ {and}
\begin{eqnarray}
Q_\mu(x;\alpha,
\beta,N)&=&{(-1)^\mu\Gamma(\beta+\mu+1)\Gamma(-N)\Gamma(\alpha+1)\Gamma(N-x)\Gamma(\alpha+\mu+x+1)\over
\Gamma(-N+\mu)\Gamma(\alpha+x+1)\Gamma(\alpha+\mu+1)\Gamma(N-\mu-x)\Gamma(\beta-\mu+1)}\nonumber\\
&&\times\sum_{j=0}^{+\infty}{(-\mu)_j(-x)_j(\beta+N-x+1)_j\over
j!(N-\mu-x)_j(-\alpha-\mu-x)_j}\nonumber\\
&=&{(-1)^\mu\Gamma(\beta+\mu+1)\Gamma(-N)\Gamma(\alpha+1)\Gamma(N-x)\Gamma(\alpha+\mu+x+1)\over
\Gamma(-N+\mu)\Gamma(\alpha+x+1)\Gamma(\alpha+\mu+1)\Gamma(N-\mu-x)\Gamma(\beta-\mu+1)}\nonumber\\
&&\times\hypergeom{3}{2}{-\mu,-x,\beta+N-x+1}{N-\mu-x,-\alpha-\mu-x}{1}.\nonumber
\end{eqnarray}
\end{proof}

\section*{Acknowledgement}
The first author would like to thank TWAS and DFG for their support to sponsor a research visit at the Institute of Mathematics of the University of Kassel in 2015 under the reference number 3240278140, where part of this work has been written. 


\end{document}